\documentclass[12pt]{article}
\usepackage{amsthm,amsmath,amssymb,amsthm}
\usepackage{dsfont}

\let\al\alpha
\newcommand\C{\mathbb C}
\newcommand\D{\mathbb D}
\let\de\delta
\newcommand{\E}{\mathbb{E}}
\let\ep\varepsilon

\newcommand\PP{\mathcal P}
\newcommand{\prob}{\mathbb{P}}

\newcommand\ro{\rho}
\newcommand{\R}{\mathbb{R}}
\newcommand\RR{\mathcal R}
\newcommand\SSS {\mathcal S}
\newcommand\TT{\mathcal T}

\newcommand\YY{\mathcal Y}
\newcommand\Z{\mathbb Z}
\let\le\leqslant
\let\ge\geqslant

\DeclareMathOperator*{\argmin}{arg\,min}

\newtheorem{theorem}{Theorem}[section]

\newtheorem{lemma}[theorem]{Lemma}
\newtheorem{remark}[theorem]{Remark}
\theoremstyle{definition}

\begin{document}

\title{Estimates of $mm$-entropy of a stable L\'evy process\footnote{Work supported by grant no. 075-15-2025-013
of the Ministry of Education of Russia}}
\author{V.O. Khamzin, M.A. Lifshits}

\maketitle
\hfill {\bf Dedicated to the 80-th anniversary of H.\,Wo\'zniakowski}
\bigskip

\begin{abstract}
In the article the $mm$-entropy (an entropy of a metric measure space) introduced by C.\,Shannon is evaluated for an $\alpha$-stable L\'evy process. For $\alpha<1$ the double-sided
estimates of the same order are obtained for process distribution in Skorokhod space.
\end{abstract}

\section{Introduction}

\subsection{$mm$-entropy: definition and history}
Various kinds of entropy are used for complexity evaluation of mathematical objects since Clausius classical works in termodynamics \cite{Rudolf}.
In this article, entropy is considered as a complexity measure for metric spaces. The most known  notion used for this purpose is metric entropy some times also called $\ep$-entropy. Let us recall the corresponding definition.

Let $(\mathcal{X}, d)$ be a metric space and let $A\subset \mathcal{X}$ be a totally bounded subset of $\mathcal{X}$.
Let $N(A, \ep)$ denote the minimal number of closed balls of radius $\ep$ sufficient to cover the set $A$. Then
\[
   H(A, \ep) := \ln N(A, \ep)
\]
is called  \textit{metric entropy} of $A$.

The idea of using metric entropy for measuring the  "complexity" and  "massiveness" of the sets in metric spaces dates back
to the works of L.S.\,Pont\-ryagin and L.G.\,Shnirelman from the thirties years of the XX century. A huge progress in the study of
metric entropy of various sets in functional spaces was achieved by A.N.\,Kolmogorov's school, see \cite{Kolmogorov} and the famous surveys \cite{KolmogorovTihomirov,Tih1,Tih2, Tih3}.

A bit later, the metric entropy found unexpected and remarkable applications in Probability Theory where R.\,Dudley, V.N.\,Sudakov and X.\,Fernique used it for investigation of the sample path properties of Gaussian processes, cf. \cite{Lifshits1, Lifshits2}.

The notion of metric entropy can be naturally extended to other classes of objects. For a {\it metric space} $(\mathcal{X},\rho)$ equipped with a {\it finite Borel measure} $P$  let us define the {\it $mm$-entropy} by
\begin{eqnarray*}
    N^{mm}(\ep, \delta) &:=& \min\{n: \exists x_1, x_2\dots, x_n: P(\mathcal{X}\setminus \cup_{i = 1}^n B(x_i, \ep)) \le \delta\},
\\
    H^{mm}(\ep,\delta) &:=& \ln N^{mm}(\ep, \delta),
\end{eqnarray*}
where $B(x, r)$ is the closed ball of radius $r$ and center $x$. Notice that this definition uses the well known fact
that a measure $P$  is concentrated on a compact up to an arbitrarily small mass \cite{Halmosh}. The index "$mm$" stands for "measure metric"  by analogy with M.\,Gromov's
$mm$-spaces \cite{Gromov}. The quantity $H^{mm}(\cdot, \cdot)$ is called \textit{$mm$-entropy} of the triplet $(\mathcal{X},\rho,P)$. See also \cite{VerVepZat} for a more detailed analysis of the structure "distance plus measure".

The interest to the study of $mm$-entropy is, in particular, based on the fact that it enables to define a non-trivial invariant of $mm$-spaces with respect to measure preserving isometries. For example, one can state that there is no such isomorphism of two $mm$-spaces if their $mm$-entropies are different.

Let us say some words about the curious history of this notion. In fact, it was introduced by C.\,Shannon in his classical work \cite[Appendix 7]{Shannon} but passed completely unappreciated on the background of other aspects of this work which made a revolution in information transmission theory. It is enough to say that in the first Russian translation of \cite{Shannon} the related Appendix (as well as some others) was omitted as being of no interest.
Therefore, the notion of $mm$-entropy, unlike that of metric entropy, has been almost not studied.
We are only aware of the work \cite{PinSof} and some other works of the same authors where this object appears in the context of signal transmission theory. Some researchers, apparently unaware of \cite{Shannon}, rediscovered it.

In the literature, several well studied notions can be found that are close to $mm$-entropy as
tools for measuring an $\ep$-discretization error but have an integral form. In a vast literature devoted to the dictionary based information transmission, the moments of the discretization error are used, see, for example the works of G.\,Luschgy and G.\,Pages \cite{LuschPa04},
S.\,Dereich with coauthors \cite{Der_dis}--\cite{DerLif}. On the other hand, A.M.\,Vershik used the Kantorovich distance in his works from ergodic theory \cite{Ver10}.

In recent years, with A.M.\,Vershik's initiative, the interest to $mm$-entropy has been renewed due to potential applications in ergodic theory.
Here, one of the most natural objects of study are the distributions of random processes in the corresponding spaces of trajectories.
In this direction, A.M.\,Vershik and M.A.\,Lifshits considered in \cite{VerLif} the distributions of Gaussian processes. They showed
that for a wide class of Banach spaces equipped with a Gaussian measure $mm$-entropy is tightly related to the metric entropy of the corresponding dispersion ellipsoid (the unit ball of the measure's kernel, RKHS) and to the small ball measures. Interestingly, the result of \cite{VerLif} is new even for the distribution of a Wiener process.

The present work is the first one where the $mm$-entropy is evaluated for an important class of {\it non-Gaussian} processes.

\subsection{Stable L\'evy processes and Skorokhod space}

Let us describe the metric space and the measure that we  are going to investigate.

Consider an {\it $\alpha$-stable L\'evy process} $X(t), \ t\in[0, 1]$ with parameter $0 < \alpha < 1$. It is well known that these processes admit a representation
\begin{equation*}
        X(t) = 
        a t + \int_{\mathcal{R}}\mathds{1}_{\{s \leq t\}}u N(du, ds),
\end{equation*}
where $a\in \R$ is a drift coefficient, $\mathcal{R} = \R \times \R_{+}$,  and $N$ is a Poisson random measure on
$\mathcal{R}$ with intensity measure $\nu(du)ds$, while the L\'evy measure $\nu$ has a form
\begin{equation}\label{common: measure nu}
      \nu(du) = \Big(\mathds{1}_{\{u>0\}}\frac{d_1}{u^{1 + \alpha}} + \mathds{1}_{\{u<0\}}\frac{d_2}{|u|^{1+\al}}\Big)du, \qquad  d_1, d_2\geq 0.
\end{equation}
This integral representation essentially means that the number of jumps of the process $X$ that occur on a time
interval $ds$ and whose size belongs  to an interval $du$ is a Poisson random variable of intensity $\nu(du)ds$.
Therefore, $X$ is a jump process with independent and homogeneous in time increments.
If $a=0$, i.e. there is no trend, then the process $X$ is also $1/\alpha$-self-similar.

\medskip

\bigskip

The distribution of $X$ is usually considered in the {\it Skorokhod space} (see more details in \cite{Skor1, Skor2, Bil, Whitt}); we recall now briefly its definition.
\medskip

A function $f:[0, 1] \mapsto \mathbb R$ belongs to the class of \textit{c\`adl\`ag} functions, if it is right continuous and has all left limits, i.e. there exist the limits
\begin{equation*}
            \begin{gathered}
                1)\lim_{s \to t+}f(s) = f(t), \quad  t\in[0,1),\\
                2)  \lim_{s \to t-} f(s), \quad  t\in(0,1] .
            \end{gathered}
\end{equation*}
Such functions form a space $\D[0,1]$ called \textbf{\it Skorokhod space}.

One can define several reasonable distances on $\D[0,1]$. We will be interested in one of the most frequently used of them, the so called $J$-distance.
The Skorokhod's  $J$-distance $\ro$ between two c\`adl\`ag functions $f, g$ is defined as
\[
            \begin{gathered}
                    \ro(f, g) := \inf_{\theta \in \Theta}
                    \max \big[
                            \sup_{0 \leq t\leq 1} |f(t) - g(\theta(t))|, 
                            \sup_{0\leq t\leq 1}|t-\theta(t)
                        |\big],
            \end{gathered}
\]
where $\Theta$ is the class of all increasing continuous mappings $[0,1]$ onto itself.

Although $\D[0,1]$ is both a linear and a topological space, it does not belong to the class of linear topological spaces because the addition operation is not continuous in it. This fact creates certain technical difficulties for operating in it, cf. Lemma \ref{mainlemma1} below.
\bigskip

For our process $X$, as well as for other jump processes, the space $\D[0, 1]$ is a natural sample path space in a sense that there exists a version of $X$ such that its sample paths $X(t), 0\le t \le 1,$ belong to $\D[0,1]$ with probability one.
Accordingly, there is a distribution of $X$ on $\D[0, 1]$, i.e. the Borel measure $\PP$ defined as
\[
    \PP(A):= \prob(X(\cdot)\in A).
\]

In the following, we consider the Skorokhod space equipped with the $J$-distance and the distribution $\PP$ of an
$\alpha$-stable L\'evy process $X$,
as a metric triplet $(\mathcal{X}, \rho, P)$, for which the $mm$-entropy will be evaluated.

\subsection{The main result}

For $mm$-entropy of the just defined triplet $(\D[0,1],\rho,\PP)$  we will prove the following double-sided estimates having the same order.

\begin{theorem} \label{th1}
For every $\de > 0$ there exist positive  constants $B_-,B_+$ such that for all sufficiently small $\ep>0$ it is true that
\begin{equation} \label{th1_bounds}
          B_- \, \ep^{-\alpha} \,  |\ln\ep|  \leq H^{mm}(\ep, \de) \leq B_+ \, \ep^{-\alpha} \, |\ln\ep|.
\end{equation}
\end{theorem}

The article is organized as follows. In Section \ref{s:lower} a lower bound announced in Theorem \ref{th1} is obtained. In Section \ref{s:upper} we give a sketch of the proof of the corresponding upper bound using some auxiliary results whose proofs are postponed to 
Section \ref{s:lemmasproofs}. The final section describes further possible research directions.
\bigskip

In the following we use some standard notations. For a finite set $A$, the
quantity $|A|$ denotes the number of elements of $A$. We denote by $||x||_\infty$ the standard
sup-norm of a function $x$ on $[0,1]$.


\section{The lower bound for $mm$-entropy}
\label{s:lower}
   
Let $x\in \D$ 
and $\ep>0$.  We denote:
\\ $M_\ep(x)$ -- the number of jumps of $x$ with absolute value larger than $5\,\ep$,
\\ $K_\ep(x)$ --  the number of jumps of $x$ with absolute value larger than $3\,\ep$,
\\ $N_\ep(x)$ -- the number of jumps of $x$ with absolute value larger than $\ep$.

\begin{remark}
If $x,y\in\D$  are such that $\ro(x,y) \le \ep$, then
\begin{equation}\label{lowerbound: remark1}
        N_\ep(y) \ge K_\ep(x) \ge M_\ep(y).
\end{equation}
\end{remark}

For a non-negative integer $n$ let us consider the set of trajectories 
$\{z\in\D: N_\ep(z) = n\}$, 
i.e. the trajectories having exactly $n$ jumps with absolute value larger than $\ep$. 
We will evaluate the measure of the intersection of this set with a given closed ball of radius
$\ep$.
    
\begin{lemma} \label{lemma1}
For every $n \ge 0$ and every $x\in \D$ it is true that
\begin{equation*}
    \PP(B_{x}(\ep) \cap \{z: N_\ep(z) = n\}) 
    \le \PP(\{z:N_\ep(z) = n\})   \cdot n^k \cdot (2\ep)^k,
\end{equation*}
where $k = K_\ep(x)$, $B_x(\ep)$ is the closed ball of radius $\ep$ centered at $x$. 
\end{lemma}

\begin{proof} [Proof of Lemma \ref{lemma1}]

If the intersection of  the ball $B_x(\ep)$  with $\{z : N_\ep(z) = n\}$ is empty, 
then the claim is obviously true.

Assume that this intersection is non-empty. Take $z$ such that
$\ro(x, z) \le \varepsilon $ and  $N_\ep(z) = n$.

Denote by $(t_1, t_2, \dots ,t_k)$ the instants of jumps of the function $x$ with absolute value larger than
$3 \ep$. Let $(s_1, s_2,\dots ,s_n)$ be the instants of jumps of the function $z$ 
with absolute value larger than $\ep$. Since  $n$ is fixed, we may assume that $s_1, s_2,\dots, s_n$ are independent and uniformly distributed on $[0,1]$.
    
 Since $\ro(x, z) \le \ep$, every jump of $x$ with absolute value larger than $3 \ep$ 
 is associated with a jump of $z$, with absolute value larger than $\ep$ and the difference of time
 instants of these jumps does not exceed $\ep$. This means that there exists a subset of indices 
 $ m_1, m_2, \dots, m_k $ such that
    
\[
    \max_{1\le i\le k} |t_i - s_{m_i}| \le \ep.
\]

Let us evaluate the $\PP$-measure of such $z$. First, let us fix     
$ \{m_1, m_2, \dots, m_k\}$. Since  the instants $s_{m_1}, s_{m_2}, \dots, s_{m_k}$ are independent and uniformly
distributed, the $\PP$-measure  of $z$ such that
\[
    \quad s_{m_i} \in [t_i - \varepsilon, t_i + \varepsilon] \qquad   i \le k, 
\]
admits the bound
$
    \PP(\{z : N_\varepsilon(z) = n\}) \cdot (2\varepsilon)^k.
$
It remains to recall that the number of possible subsets of indices    
$\{ m_1, m_2, \dots, m_k \}$ does not exceed $n^k$, which leads to the announced inequality.  

\end{proof}


For $\ep > 0$  let $G(\ep)$ denote the expected number of jumps of the process $X$ on $[0,1]$
with absolute value exceeding $\ep$.
By the definition of $X$ and the formula for its intensity measure \eqref{common: measure nu} 
one has
\begin{equation}\label{common:G(l)}
       G(\ep) =    \int_{|u|> \ep}\nu(du )
       =(d_1+d_2)\, \alpha^{-1} \, \ep^{-\alpha} =G(1) \, \ep^{-\alpha}.
\end{equation}
The random variable $M_\ep(X)$ has Poisson distribution with intensity parameter
\[  
   \mu_{M,\ep} := G\big(5\ep\big),
\] 
By using \eqref{common:G(l)}, this can be rewritten as 
\[
   \mu_{M,\ep} = G(1)\big(5\ep\big)^{-\alpha}.
\]
Let
\[
     T_{M,\ep} := \left\{x\in \D: |M_\ep(x) - \mu_{M,\ep}| \le \sqrt{\frac{6\mu_{M,\ep}}{\delta}}\right\}.
\]
By Chebyshev inequality, we have 
\begin{equation}\label{lowerbound: condition on M}
    \prob\{X\not\in T_{M,\ep}\} \le \frac{\de}{6}.
\end{equation}
Similarly, let us define two other sets of functions and write down the corresponding inequalities:
\[
    T_{K,\ep} := \left\{x\in \D: |K_\ep(x) - \mu_{K,\ep}| \le \sqrt{\frac{6\mu_{K,\ep}}{\delta}}\right\},
\]
with $\mu_{K,\ep} = G(1)\big(3\ep\big)^{-\alpha}$,  
\begin{equation}\label{lowerbound: condition on K}
    \prob\{X\not\in T_{K,\ep}\} \le \frac{\de}{6},
\end{equation}
and
\[
   T_{N,\ep} := \left\{x\in \D: |N_\ep(x) - \mu_{N,\ep}| \le \sqrt{\frac{6\mu_{N,\ep}}{\delta}}\right\},
\]
where $\mu_{N,\ep} =G(1)\ep^{-\alpha}$, 
\begin{equation}\label{lowerbound: condition on N}
    \prob\{X\not\in T_{N,\ep}\} \le \frac{\de}{6},
\end{equation}

We call  $x\in\D$  \textit{typical}, if it belongs to 
$T_\ep:=T_{M,\ep}\cap T_{K,\ep}\cap T_{N,\ep}$ 
and  \textit{untypical} otherwise. 
By inequalities \eqref{lowerbound: condition on M}, \eqref{lowerbound: condition on K}, \eqref{lowerbound: condition on N}, the $\PP$-measure of untypical functions does not exceed $\frac{\de}{2}$. 

Therefore, we need to cover a part of the set $T_\ep$ having $\PP$-measure at least
$1 - \frac{3\de}{2}$.
 
Let us evaluate the $\PP$-measure of typical functions covered by a ball of radius $\ep$. 
Consider an arbitrary $x\in \D$ and assume that the ball $B_{x}(\ep)$ contains a typical $y\in T_\ep$. 
Let $m_\ep:= \mu_{M,\ep} - \sqrt{\frac{6\mu_{M,\ep}}{\delta}}$. 
Since $y \in T_\ep$, we have 
\[
    M_\ep(y) \ge m_\ep,
\]
hence, by \eqref{lowerbound: remark1},
\begin{equation}\label{lowerbound: inequality on K_ep}
    K_\ep(x) \ge M_\ep(y) \ge m_\ep.
\end{equation}

By Lemma \ref{lemma1} it is true that
\begin{eqnarray}  \nonumber
     \PP\big(B_{x}(\ep)\cap\{z: N_\ep(z) = N_\ep(y)\}\big)
     &\le&
     \PP\big(\{z:N_\ep(z) = N_\ep(y)\}\big)\cdot N_\ep(y)^{K_\ep(x)}\cdot (2\ep)^{K_\ep(x)} 
\\  \label{lower_bound: almost_final_1}
     &=&  \PP\big(\{z:N_\ep(z) = N_\ep(y)\}\big)\cdot (2N_\ep(y)\ep)^{K_\ep(x)}.
\end{eqnarray}
Since $y\in T_\ep\subset T_{N,\ep}$, we have
\[
    N_\ep(y) \le \mu_{N,\ep} + \sqrt{\frac{6\mu_{N,\ep}}{\delta}},
\]
therefore,
\[
    N_{\ep}(y) \le \tilde c \cdot \ep^{-\alpha},
\]
where $\tilde c$ does not depend on $\ep$. 
By \eqref{lowerbound: inequality on K_ep} it follows  that for small $\ep$ one has
\begin{equation}\label{lowerbound: inequality on 2N_ep ep}
    \big(2N_\ep(y)\ep\big)^{K_\ep(x)} \le (2\tilde c \cdot \ep^{1-\alpha})^{K_\ep(x)} \le (2\tilde c\cdot \ep^{1-\alpha})^{m_\ep}.
\end{equation}

We stress that this bound only makes sense for $\alpha<1$, as well as its corollaries that will be derived now.
From \eqref{lower_bound: almost_final_1} and \eqref{lowerbound: inequality on 2N_ep ep},
it follows that
\begin{equation*}
    \PP\big(B_{x}(\ep)\cap\{z: N_\ep(z) = N_\ep(y)\}\big) \le \PP\big(\{z:N_\ep(z) = N_\ep(y)\}\big)\cdot (2\tilde c\cdot \ep^{1-\alpha})^{m_\ep}.
\end{equation*}      

By summing up these estimates over $n\in N_\ep(T_\ep):= \{N_\ep(x), x\in T_\ep\}$, we obtain
\begin{eqnarray*} 
      \PP\big(B_x(\ep)\cap \ep\big)
      &\le& \sum_{n \in N_\ep(T_\ep)}\PP\big(B_{x}(\ep)\cap\{z: N_\ep(z) = n\}\big)
\\   
      &\le& \sum_{n \in N_\ep(T_\ep)} \PP\big(\{z: N_\ep(z) = n\}\big)
          \cdot(2\tilde c\cdot\ep^{1-\alpha})^{m_\ep}
\\   
      &\le& (2\tilde c\cdot\ep^{1-\alpha})^{m_\ep}
          \cdot\sum_{n \in N_\ep(T_\ep)}
               \PP\big(\{z: N_\ep(z) = n\}\big) 
       \le (2\tilde c\cdot \ep^{1-\alpha})^{m_\ep}.
\end{eqnarray*}

Therefore,
\begin{equation*}
    N^{mm}(\ep, \de) \ge \frac{1 - \frac{3\de}{2}}{(2\tilde c\cdot \ep^{1-\alpha})^{m_\ep}}.
\end{equation*}

Hence
\begin{equation}\label{lowerbound: final inequality}
    \begin{gathered}
        H^{mm}(\ep,\de)\ge m_\ep |\ln(\ep^{1-\alpha})|\big(1 + o(1)\big) = \mu_{M,\ep}(1-\alpha)|\ln\ep|\big(1+o(1)\big). 
    \end{gathered}
\end{equation}

Since $\mu_{M,\ep} = G(1)(5\ep)^{-\alpha}$, 
the bound \eqref{lowerbound: final inequality} yields 
\begin{equation*}
    H^{mm}(\ep, \de) \ge B_-\ep^{-\alpha}|\ln\ep|,
\end{equation*}
for all small $\ep$ and some $B_->0$ not depending on $\ep$.


\section{The upper bound for $mm$-entropy} 
\label{s:upper}

Let us split  the process $X$ into two parts:
\begin{equation} \label{split}
     X = X_\ep + X^\ep,
\end{equation}
where
\begin{eqnarray*}
                X^\ep(t) &:=& at + \int_{\RR}\mathds{1}_{\{s \le t\}}\mathds{1}_{|u| > \ep}u N(du, ds),
\\
                X_\ep(t) &:=& \int_{\RR}\mathds{1}_{\{s \le t\}}\mathds{1}_{|u| \le \ep}u N(du, ds).
\end{eqnarray*}
The process $X^\ep$ contains the linear part  and the large jumps of absolute value exceeding $\ep$, 
while the process $X_\ep$  contains other (small) jumps.

For a function $f$ defined on $[0,1]$ and a number $\Delta \ge 0$ let us introduce
the {\it oscillation modulus}
\[ 
        \omega(f, \Delta) := \sup_{|s - t| \le \Delta} |f(s) - f(t)|.
\]
        
\begin{lemma}\label{mainlemma1}
For all $x_1,x_2,y_1, y_2\in \D$ it is true that
\[ 
    \ro(x_1 + x_2, y_1 + y_2) 
    \le \ro(x_1, y_1) + ||x_2 - y_2||_\infty + \omega\big(y_2, 2\ro(x_1, y_1)\big).
\]
\end{lemma}
    
The result of \ref{mainlemma1} enables to handle the processes of large and small jumps separately. 
The proof of this lemma  is given in Section \ref{s:lemmasproofs}.

Let us now state two basic lemmas concerning the sizes of approximating nets.
Their proofs are also given in Section \ref{s:lemmasproofs}.

\begin{lemma}\label{mainlemma2}
    For all $B > 0$ and $\de > 0$ there exists a constant $K_1>0$ such that for all small $\ep$ 
    one can find a finite net $\YY^\ep$ of size not exceeding $\exp(K_1\ep^{-\alpha}|\ln \ep|)$
    and satisfying
\begin{equation} \label{lemma2_prob_bound}
            \prob\big( \min_{y^\ep\in \YY^\ep}\ro(X^\ep, y^\ep) > \ep^B\big) < \de.
\end{equation}
\end{lemma}
    
\begin{lemma}\label{mainlemma3}
    For all $B>2\alpha$ and $\de > 0$ there exists a constant $K_2 >0$ such that for all small $\ep$ 
    one can find a finite net $\YY_\ep$ of size not exceeding $\exp(K_2\ep^{-\alpha}|\ln \ep|)$ and satisfying
\begin{equation} \label{maxomega}  
       \max_{y \in \YY_\ep } \   \omega(y,2\ep^{B}) \le   2\ep
\end{equation}
and    
\begin{equation*}
            \prob\big( \min_{y\in \YY_\ep} ||X_\ep - y||_\infty > 3 \ep\big) \le \de.
\end{equation*}
\end{lemma}
\bigskip

Let us derive the required upper bound for $mm$-entropy from these lemmas.
Let us fix a sufficiently large $B$, namely $B>\max\{1,2\alpha\}$.

Let us consider a net $\YY^\ep$ from Lemma \ref{mainlemma2} corresponding to the parameter $B$ 
and to the covering measure $\frac{\de}{3}$. Define a random element $y^\ep\in\YY^\ep$ by the relation
\begin{equation*}
        y^\ep := \argmin_{y\in \YY^\ep}\ro(X^\ep, y) .
\end{equation*}

Similarly, let us consider a net $\YY_\ep$  from Lemma \ref{mainlemma3} with the same parameters
and define a random element
\begin{equation*}
        y_\ep := \argmin_{y\in \YY_\ep}  ||X_\ep - y||_\infty .
\end{equation*} 

 We will show that the net $\YY^\ep + \YY_\ep$ approximates the process $X$ with the required precision. 

Assume first that the following inequalities are true 
\[
   \ro(X^\ep, y^\ep) \le \ep^B,   \qquad  ||X_\ep - y_\ep||_\infty \le 3\, \ep.
\]
Then by Lemma \ref{mainlemma1} we have
\begin{eqnarray*}
   \ro(X^\ep + X_\ep, y^\ep+ y_\ep) 
   &\le& \ro(X^\ep, y^\ep) +  ||X_\ep - y_\ep||_\infty +  \omega(y_\ep,2  \ro(X^\ep, y^\ep) )
\\
  &\le&  \ep^B  +  3 \, \ep  + \omega(y_\ep,2\ep^B).
\end{eqnarray*}
    
By taking into account the inequality $B>1$ and the property \eqref{maxomega}, 
we obtain in this case
\[
    \ro(X^\ep + X_\ep, y^\ep+ y_\ep) \le \ep +  3\,\ep  +2\,\ep = 6\, \ep.
\]
By Lemmas \ref{mainlemma2}  and \ref{mainlemma3} we obtain the probabilistic estimate
\begin{eqnarray*}
     && \prob\big(\ro(X^\ep + X_\ep, y^\ep+ y_\ep) > 6\,\ep \big)
\\
    &\le& \prob\big(\ro(X^\ep, y^\ep) > \ep^B \big) 
            + \prob\big(||X_\ep - y_\ep||_\infty > 3 \ep\big) 
    \le   \frac{\de} 3 +   \frac {\de} 3  < \de.
\end{eqnarray*}
    
The size of the net $\YY^\ep + \YY_\ep$ does not exceed
\begin{equation*}
        \exp(K_1 \, \ep^{-\alpha}|\ln\ep| + K_2\, \ep^{-\alpha}|\ln\ep|) 
        := \exp(K_3\, \ep^{-\alpha}|\ln\ep|),
\end{equation*}
hence
\begin{equation*}
        H^{mm}\Big(6\ep, \de\Big) \le K_3 \, \ep^{-\alpha}|\ln\ep|,
\end{equation*}
Replacement $ 6\,\ep$ to  $\ep$ yields
\begin{equation*}
        H^{mm}\big(\ep, \de\big) \le B_+ \, \ep^{-\alpha}|\ln\ep|,
\end{equation*}
which is the upper bound required in Theorem \ref{th1}.


\section{Proofs of lemmas} 
\label{s:lemmasproofs}

\subsection{Proof of Lemma \ref{mainlemma1}}
If $\ro(x_1,y_1)=0$, i.e. $x_1=y_1$, then the lemma's claim follows from the trivial estimate
\[
  \ro(x_1+x_2,x_1+y_2) \le || (x_1+x_2) -(x_1+y_2)||_\infty =  || x_2 -y_2||_\infty.
\]
Therefore, we proceed under assumption $\ro(x_1,y_1)>0$.
    
By the definition of $\ro(x_1,y_1)$, for every $\zeta > 0$ there exists a variable change
$\theta: [0,1]\mapsto[0,1]$ such that
\[
     |x_1(t) - y_1(\theta(t))| \le \ro(x_1, y_1) + \zeta, \qquad t\in[0,1],
\] 
and
\[ 
     |t - \theta(t)|\le \ro(x_1,y_1)+\zeta,      \qquad t\in[0,1].
\]
It follows that under the same time change,
\begin{eqnarray*}
  &&   |\big(x_1 + x_2\big)(t) - \big(y_1 + y_2\big)(\theta(t))| 
     \le |x_1(t) - y_1(\theta(t))| + |x_2(t) - y_2(\theta(t))|
\\
     &\le& \ro(x_1, y_1) + \zeta + |x_2(t) - y_2(t)| + |y_2(t) - y_2(\theta(t))|
\\
     &\le& \ro(x_1, y_1) + \zeta + ||x_2 - y_2||_\infty + \omega\big(y_2, \ro(x_1, y_1) + \zeta\big).
\end{eqnarray*}

Since the function $\omega(y_2,\cdot)$ is non-decreasing, for $\zeta < \ro(x_1,y_1)$ we have
\[
    ||\big(x_1 + x_2\big) - \big(y_1 + y_2\big)(\theta(\cdot))||_\infty
    \le \ro(x_1, y_1) + \zeta + ||x_2 - y_2||_\infty + \omega\big(y_2, 2\ro(x_1, y_1)\big),
\]
and letting $\zeta$ go to zero we obtain the lemma's claim.


\subsection{Proof of Lemma \ref{mainlemma2}}
        
Recall that the process $X^\ep$ has a representation
\[
    X^\ep(t) = at + \int_{\RR}\mathds{1}_{\{s \le t\}}\mathds{1}_{|u| > \ep}u N(du, ds).
\]  
       
Note that the L\'evy measure of $X^\ep$ is finite, i.e. it is a sum of a linear function and a compound Poisson process, admitting another representation
\begin{equation*}
            X^\ep(t) = at + \sum_{i = 1}^{K_\ep(t)} V_\ep^i,
\end{equation*}
where $K_\ep(t)$ is a Poisson process of intensity $G(\ep)$ defined in \eqref{common:G(l)} and
$\{V_\ep^i\}_{i = 1}^{\infty}$ are i.i.d. random variables with common distribution
$\PP_\ep(du) = \frac{\mathds{1}_{|u| > \ep}\nu(du)}{G(\ep)}$ 
(the specific form of $\PP_\ep$ plays no role in subsequent estimates).

We may identify the space of sample paths of the process $X^\ep$ with the set 
\begin{equation*}
            \begin{gathered}
                \TT := \bigsqcup_{n = 0}^{\infty} \R^n\times E_n,
            \end{gathered}
\end{equation*}
where $E_n$ is the $n$-dimensional simplex defined by inequalities
\[
   0 \le t_1 \le t_2 \ ...\le t_n \le 1.
\]
    
An element $(v, t)$, where $v\in \R^n$ and $t \in E_n$, is identified with the step function having
jumps of size $v_1, v_2, \dots, v_n$ at the time instants $t_1, t_2, \dots, t_n$,
respectively.

The distance $\ro$ can be naturally transferred from $\D$ to $\TT$, too.
        
The distribution $Q_\ep$ of the process $X^\ep$ on $\TT$ has the following form:
\begin{equation*}
            \begin{gathered}
                Q_\ep|_{R^n \times E_n} 
                = e^{-G(\ep)} \, \frac{G(\ep)^n}{n!} \cdot \underbrace{\PP_\ep \times \PP_\ep \times ... \times \PP_\ep}_\text{$n$ times} \times I_n,
            \end{gathered}
\end{equation*}
where $I_n$ is the uniform distribution on $E_n$.
\medskip

We will need the following property of the distance $\ro$ in terms of the space $\TT$.

\begin{lemma}\label{sublemma}
    Let $(v, t)$ be an element of $\R^n\times E_n$ and let $r>0$. Then the cube with center 
    at $(v,t)$ in $\R^n\times E_n$ and edge length $\tfrac{r}{4(n + |a|)}$ is contained in the 
    $\ro$-ball with the same center and radius $r$.
\end{lemma}

\begin{proof}[Proof of lemma \ref{sublemma}]
Consider an arbitrary point $(u,s)\in \R^n\times E_n$ belonging to the cube introduced in the lemma,
i.e.
\[
    \max_{1\le i\le n} \ \max\{ |v_i-u_i|, |t_i-s_i| \} \le \frac{r}{4(n + |a|)}.
\]
Denote by $f$ and $g$ the sample paths from $\D$ corresponding to the points $(v,t), (u,s)$, respectively.
            
Consider the piecewise linear mapping $\theta$ of the interval $[0,1]$ onto itself such that 
$\theta(0) = 0$, $\theta(1) = 1$ and $\theta(t_i) = s_i$, $1\le i\le n$. 
It maps the instants of jumps of $f$ in those of $g$. 

Since the difference of jump sizes  $v_i$ and $u_i$ is bounded by $\frac{r}{4n}$
and the total number of jumps is equal to $n$, we have 
\begin{equation}\label{covering: sup}
                \sup_{t\in[0, 1]} |f(\theta(t)) - g(t)| \le \frac{r}{4n} \cdot n 
                 + |a|\cdot\sup_{t\in[0,1]}|t - \theta(t)| \le \frac{r}{4} + \frac{r}{4} = \frac{r}{2}.    
\end{equation}
            
By the the definition of $\ro$ and \eqref{covering: sup}, we have  
\[
   \ro(f, g) \le \max\big\{\sup_{t\in[0, 1]} |f(\theta(t)) - g(t)|, \sup_{t\in[0,1]}|t - \theta(t)|\big\} 
   \le \max\big\{\frac{r}{2}, \frac{r}{4n}\big\} < r,
\]
i.e. $g$ belongs to $\ro$-ball of radius $r$ with center at $f$.
\end{proof}
    
Since the measures of spaces $\R^n\times E_n$ obey Poisson law with parameter $G(\ep)$, 
one may choose $\beta_\delta >0$ such that 
\begin{equation}\label{covering: threshold}
      Q_\ep \left(
          \bigcup_{ n:\, |\frac{n}{G(\ep)} - 1|>\beta_\delta} \R^n\times E_n
          \right) 
          < \frac{\de}{2}.
\end{equation}
    
By Lemma \ref{sublemma}, every cube in  $\R^n\times E_n$ with dimension
$n \le G(\ep)(1 + \beta_\de)$ and edge length
\[
    \Delta_\ep := \frac{\ep^B}{4\big(G(\ep)(1 + \beta_\de) + |a|\big)}
\]
is contained in a $\ro$-ball of radius $\ep^B$. 

Therefore, if we will be able to cover a considerable part of the set
\[
     \bigcup_{n:\, |\frac{n}{G(\ep)} - 1|\le \beta_\delta} \left(\R^n\times E_n\right)
\]     
by a collection $\SSS$ of cubes of edge length $\Delta_\ep$ so that 
\begin{equation}\label{covering: inequality on set of squares}
            Q_\ep\Big(\bigcup_{n:\, |\frac{n}{G(\ep)} - 1| \le \beta_\delta} 
               \left(\R^n\times E_n \right)
            \setminus 
            \bigcup_{S\in \SSS} S \Big) \le \frac\de2,
\end{equation}
then, in view of \eqref{covering: threshold} and \eqref{covering: inequality on set of squares}, 
the measure of uncovered part of the space $\TT$ is bounded by 
\[
   Q_\ep\left(\bigcup_{n:\, |\frac{n}{G(\ep)} - 1|\le \beta_\delta} 
   \left( \R^n\times E_n\right) \setminus
   \bigcup_{S\in \SSS} S\right)  
    + 
    Q_\ep\left(\bigcup_{n:\, |\frac{n}{G(\ep)} - 1|>\beta_\delta} \left(\R^n\times E_n\right)
    \right) \le \de.
\]

The $\ro$-balls of radius $\ep^B$  with the same centers cover $\TT$ at least as well,
hence these centers form a net $\YY^\ep$, satisfying \eqref{lemma2_prob_bound} and
\[
            |\YY^\ep| \le |\SSS|.
\]       

Let us choose $M\ge 1$ so large that the probability that the process  $X$  has a jump of size large than
$M$, is less than $\tfrac\de2$.

For every $n \in \big(G(\ep)(1 - \beta_\de), G(\ep)(1 + \beta_\de)\big)$ we cover the set
$[-M,M]\times E_n$ with cubes of edge length $\Delta_\ep$. Since $M\ge 1, \Delta_\ep\le 1$,
we need not more than
\[
    \left( \frac{M}{\Delta_\ep} + 1\right)^n  \left( \frac{1}{\Delta_\ep}+1 \right)^{n}
    \le \frac{(2M+2)^n}{ (\Delta_\ep)^{2n} }
\]
cubes. Therefore, for the covering of all sets from the mentioned range we need not more than
\[  
            \sum_{n\in\big((1-\beta_\de)G(\ep), (1+\beta_\de)G(\ep)\big)} \frac{(2M+2)^n}{(\Delta_\ep)^{2n}} \le N_1\frac{G(\ep)(2M+2)^{(1+\beta_\de)G(\ep)}}{(\Delta_\ep)^{2(1+\beta_\de)G(\ep)}}
\]
cubes, where $N_1:=2\beta_\delta+1$. 
The required relation \eqref{covering: inequality on set of squares}
holds by the choice of $M$.
       
Thus we have constructed a net $\YY^\ep$ such that \eqref{lemma2_prob_bound} 
and 
\[
        |\YY^\ep| = |\SSS| \le N_1\frac{G(\ep)(2M+2)^{(1+\beta_\de)G(\ep)}}{(\Delta_\ep)^{2(1+\beta_\de)G(\ep)}}.
\]

By using \eqref{common:G(l)}, we derive
\begin{eqnarray*}
    \ln|\YY^\ep| &\le&  \ln N_1 + \ln G(\ep) +
                  (1+\beta_\de) \, G(\ep) \left( \ln (2M+2) + 2|\ln \Delta_\ep| \right)
\\
              &=&    \ln N_1 + (\ln G(1) +\alpha|\ln\ep| )
\\              
              && + (1+\beta_\de)\, G(1) \, \ep^{-\alpha} \left( \ln (2M+2) + 2|\ln \Delta_\ep| \right).
\end{eqnarray*}
Next, from the definition of $\Delta_\ep$ it follows
\[
   |\ln \Delta_\ep| = \ln \frac{\ep^B}{G(\ep)} +O(1) = (B+\alpha) \, |\ln\ep| + O(1).
\]
The substitution of this estimate yields
\[
       \ln|\YY^\ep| \le
                  (1+\beta_\de)\, G(1)\, \ep^{-\alpha} 
                  \cdot 2(B+\alpha) \, |\ln\ep| \, (1+o(1)),
\]
and for $K_1> 2(1+\beta_\de)\, G(1)\, (B+\alpha)$ and for small $\ep$ we arrive at the required bound
\begin{equation*}
                \ln|\YY^\ep|\le K_1 \, \ep^{-\al} \,|\ln\ep|.
\end{equation*}

\subsection{Proof of Lemma \ref{mainlemma3}}

Since the integral over Poisson random measure 
\[
     \int_{\RR}\mathds{1}_{\{s \le t\}}u N(du, ds)
\]
is an $\alpha$-strictly stable process, it is also a $\frac 1\al$-self-similar process, i.e.,
\begin{equation}\label{lemma: equality in distribution}
            \int_{\RR}\mathds{1}_{\{s \le t\}}u N(du, ds) 
    \overset{d}{=}
           \ep\int_{\RR}\mathds{1}_{\{s \le t\cdot \ep^{-\alpha}\}}u N(du, ds).
\end{equation}
If we consider in the identity \eqref{lemma: equality in distribution} only the jumps of
size not exceeding $\ep$, we obtain 
\begin{equation}\label{mainlemma3: self-similarity}
    X_\ep(t)   \overset{d}{=}  \ep \cdot X_1(t\cdot \ep^{-\alpha}).
\end{equation}

Let us construct a uniform approximation of $X_1(t), t\in[0,\ep^{-\alpha}]$, by a random function
$Y(t)$, $t\in[0,\ep^{-\alpha}]$ taking only integer values.

Define a sequence of stopping times $(\tau_j)$ and a function $Y$ as follows. 
Let $\tau_0 := 0$, and then define $\tau_{j+1}$ by $\tau_j$ distinguishing two cases:
\begin{enumerate}
    \item[a)]  If $X_1(\tau_j)\in[k ,k + \frac12), k\in \Z$,
        then
        \[ 
            \tau_{j+1}:= \min\{t \ge \tau_j: X_1(t) \ge k+1 \text{ or } X_1(t) \le k-1\},
        \] 
         and
        \[
             Y(t):=k, \qquad \tau_j\le t < \tau_{j+1}.
        \]
    \item[b)] If $X_1(\tau_j)\in[k + \frac12 , k + 1), k\in \Z$,
        then
        \[
             \tau_{j+1}:= \min\{t \ge \tau_j: X_1(t) \ge k+2 \text{ or } X_1(t) \le k\},
        \] 
        and
        \[
              Y(t):=k + 1, \hspace{0.5cm} \tau_j\le t < \tau_{j+1}.
        \]
\end{enumerate}
        
In both cases we have 
\[ 
    |Y(t) - X_1(t)|\le 1, \qquad \tau_j\le t<\tau_{j+1},
\]
therefore,
\begin{equation}\label{mainlemma3: difference between Y and X}
    \sup_{t\ge 0} |Y(t) - X_1(t)| \le 1.
\end{equation}
        
Let us consider some properties of this construction.
\begin{enumerate}
    \item[1)] A bound on the jump size of $Y(\cdot)$.
\end{enumerate}
Since the sizes of jumps of $X_1(\cdot)$ are bounded by $1$, 
in case a) we have  $X_1(\tau_{j+1})\in [k+1, k+2)$ or $X_1(\tau_{j+1})\in(k-2, k-1]$. 
Therefore,
\begin{equation*}
            |Y(\tau_{j+1}) - Y(\tau_{j})| = |Y(\tau_{j+1}) - k| 
            \le |Y({\tau_{j+1}}) - X_1(\tau_{j+1})| + |X_1(\tau_{j+1}) - k|\le \frac12 + 2.
\end{equation*}
Since $Y$ takes only integer values, it is true that
\begin{equation} \label{jumpsize}
            |Y(\tau_{j+1}) - Y(\tau_j)| \le 2.
\end{equation}
The case b) is treated similarly.
        
\begin{enumerate}
            \item[2)] A bound on the number of jumps.
\end{enumerate}
Introduce auxiliary stopping times
\[   
    \sigma_j:=\min \big\{s>0: |X_1(\tau_j + s) - X_1(\tau_j)|\ge \frac12 \big\}.
\]
      
By construction, $\sigma_j\le \tau_{j+1} - \tau_j$, therefore, for every integer $m>0$ 
we have  
\begin{equation*}
    \prob(\tau_m\le 2\ep^{-\alpha})
    \le \prob \left(\sum_{j=0}^{m-1}\sigma_j\le2\ep^{-\alpha}\right) 
    = \prob \left(\frac 1 m \sum_{j=0}^{m-1}\sigma_j\le \frac{2\ep^{-\alpha}}{m}\right).
\end{equation*}
The variables $\sigma_j$ are i.i.d. and have a common finite expectation. 
Let $A > \frac{2}{\E\, \sigma_0}$. Then for every  $m > A\ep^{-\al}$ we have    
\begin{equation*}
    \prob\left(\frac 1 m\sum_{j=0}^{m-1}\sigma_j\le \frac{2\ep^{-\alpha}}{m} \right) 
    \le
    \prob\left(\frac 1 m\sum_{j=0}^{m-1}\sigma_j\le \frac{2}{A} \right)
    =
     \prob\left(\frac 1 m\sum_{j=0}^{m-1}\sigma_j\le \frac{2}{A\, \E\sigma_0}\, \E\sigma_0 \right). 
\end{equation*}
By the choice of $A$, the law of large numbers shows that this probability tends to zero,
as $m\to\infty$. Therefore for small $\ep$ it is true that 
\begin{equation} \label{fewsteps}
    \prob(\tau_{A\ep^{-\alpha}} \le 2\ep^{-\alpha}) 
    \le \frac{\de}{6}, 
\end{equation}
i.e., with large probability there are at most  $A\ep^{-\alpha}$ jumps of $Y$ 
on the time interval $[0,2\ep^{-\alpha}]$.

\begin{enumerate}
    \item[3)] A bound on the probability of short steps. 
\end{enumerate}
Let $\beta:=B-\alpha$, where $B$ is the parameter from the assertion of Lemma \ref{mainlemma3}.  
By Lemma's assumption it is true that $\beta>\alpha$. 
Let us evaluate the probability of the inequality $\tau_{j+1} - \tau_j \le 2\ep^\beta$, 
i.e. the approximation step is short.
We start with a trivial estimate
\begin{eqnarray} \nonumber
            \prob(\tau_{j+1} - \tau_j \le 2\ep^\beta) 
       &\le& \prob(\sigma_j\le 2\ep^\beta) = \prob(\sigma_0 \le 2\ep^\beta) 
\\ \label{mainlemma2: short jump 1}
       &=& \prob \left(\sup_{t\in[0,2\ep^\beta]}|X_1(t)|\ge \frac12 \right).
\end{eqnarray}

Next, we will apply Etemadi inequality (see \cite{Etemadi}), which is usually stated for finite sums of
independent random variables but can be trivially transferred to the right-continuous processes
with independent increments. Since $X_1$ belongs to this class, Etemadi inequality yields
\[
   \prob\left(\sup_{t\in[0,2\ep^\beta]}|X_1(t)| \ge \frac12\right)
   \le 3 \sup_{t\in[0,2\ep^\beta]} \prob \left(|X_1(t)|\ge \frac 16\right),
\]
hence, by Markov inequality,
\begin{equation}\label{mainlemma2: short jump 2}
    \prob \left(\sup_{t\in[0,2\ep^\beta]}|X_1(t)| \ge \frac12 \right)
    \le 3\sup_{t\in[0,2\ep^\beta]}\frac{\E|X_1(t)|}{\frac16} \le c\, \ep^\beta,
\end{equation}
where $c$ is some constant not depending on $\ep$.
        
By chaining inequalities \eqref{mainlemma2: short jump 1} and \eqref{mainlemma2: short jump 2}, 
we obtain
\begin{equation*}
            \prob(\tau_{j+1} - \tau_j \le 2 \ep^\beta)\le c\, \ep^\beta,
\end{equation*}
hence,
\begin{equation*}
    \prob \left(\min_{j\le A\ep^{-\al}}(\tau_{j+1}-\tau_j)\le 2\ep^\beta\right) 
    \le A \, \ep^{-\alpha} \cdot c \, \ep^{\beta} 
    = A\, c \, \ep^{\beta-\alpha}. 
\end{equation*}
Since by assumption of Lemma \ref{mainlemma3} it is true that $\beta > \alpha$, taking into account
\eqref{fewsteps},   we see that for small $\ep$ the probability to observe a step of $Y$ shorter than 
$2\ep^\beta$  on the time interval $[0,2\ep^{-\alpha})$ does not exceed $\frac{\de}{3}$.
\medskip

The constructed approximating functions $Y$ have nice properties but do not form a finite net.
Therefore, we need further discretization.  The values of $Y$ are integer by construction, thus we 
only need to discretize the jump instances.
        
Introduce the events
\begin{eqnarray*}
            &E_1& :=\{\tau_{A\ep^{-\alpha}}> 2\ep^{-\alpha}\},
\\
            &E_2&:=\{\min_{j\le A\ep^{-\al}}(\tau_{j+1}-\tau_j)> 2\ep^\beta\}.
\end{eqnarray*}
     
By the estimates given above we have
\begin{equation}\label{mainlemma3: two conditions}
            \prob(E_1\cap E_2) \ge 1 - \de.
\end{equation}

The subsequent constructions are provided only on $E_1\cap E_2$.
Let us discretize $\tau_j$ by letting 
\[   
    \tilde \tau_j:=\Big[\frac{\tau_j}{\ep^\beta}\Big] \cdot \ep^\beta, 
\]
where $[x]$ stands for the integer part of a number $x$.

Since there are no short steps, we have the following inequalities relating
jump instants and their discretized versions:
\begin{equation} \label{chaintau}
   \tilde \tau_j \le \tau_j < \tau_{j+1}-\ep^\beta \le \tilde \tau_{j+1} 
   \le \tau_{j+1} <\tau_{j+2}-\ep^\beta
   < \tilde\tau_{j+2}.
\end{equation}
In particular, we have $\tilde\tau_{j} < \tilde\tau_{j+1}$.

By the definition of $Y$, we have
\[
    Y(t) = \sum_{j=0}^{\infty}K_j\mathds{1}_{\tau_j\le t<\tau_{j+1}},
\]
where $K_j$ are some random elements of $\Z$, and by \eqref{jumpsize} it is true that 
\begin{equation}\label{jumpsize2}
    |K_{j + 1} - K_{j}| \le 2. 
\end{equation}
Introduce the function
\[ 
    \tilde Y(t):=\sum_{j = 0}^{\infty}  K_j 
                \mathds{1}_{\tilde \tau_j \le t<\tilde\tau_{j+1}}.
\]

\begin{lemma}\label{sublemma3}
    On the event $E_1\cap E_2$ the possible trajectories of $\tilde Y(t), t\in[0,\ep^{-\alpha}]$, 
    form a finite set, whose size does not exceed  $\exp(K_2\ep^{-\alpha}|\ln\ep|)$  for all small
    $\ep$ and some constant $K_2>0$ not depending on $\ep$.
\end{lemma}

\begin{proof}[Proof of Lemma \ref{sublemma3}]
Notice that on $E_1$, in order to determine  $\tilde Y(t)$, $ t\in[0,\ep^{-\alpha}]$,
it is sufficient to know the jumps' values $(K_{j+1} - K_j)_{j = 0}^{A\ep^{-\alpha}}$
and the jumps' instants $(\tilde\tau_{j})_{j=0}^{A\ep^{-\alpha}}$. 
    
According to \eqref{jumpsize2}, the number of possible sets 
$(K_{j+1} - K_j)_{j = 0}^{A\ep^{-\alpha}}$ does not exceed $5^{A\ep^{-\alpha}}$. 
    
Every jump instant  is chosen from the set $\ep^\beta \cdot \{1,2, \dots, \ep^{-\alpha-\beta}\}$, 
therefore the number of possible sets of jumps does not exceed $(\ep^{-\alpha-\beta})^{A\ep^{-\alpha}}$. 
    
It follows that for small $\ep$ the number of possible trajectories of $\tilde Y$ does not exceed
\[
      5^{A\ep^{-\alpha}} \cdot  (\ep^{-\alpha-\beta})^{A\ep^{-\alpha}} 
      \le  (\ep^{-\alpha-\beta})^{2A\ep^{-\alpha}} 
      := \exp{(K_2\ep^{-\alpha}|\ln\ep|)}.
\]      
\end{proof}

Let us now consider the quality of approximation provided by $\tilde Y$.
From \eqref{chaintau} and from the definition of $\tilde Y$, it follows that
\[
   Y(t) - \tilde Y(t) =
   \begin{cases}   0,            &    \tau_j\le t< \tilde\tau_{j+1}, \\
                   K_j- K_{j+1}, & \tilde\tau_{j+1}\le t< \tau_{j+1}.
    \end{cases}
\]
We derive from the estimate \eqref{jumpsize2} that 
\begin{equation}\label{mainlemma3: difference between Y and tilde Y}
    \sup_{t\in[0,\ep^{-\alpha}]}|Y(t) - \tilde Y(t)| \le 2.
\end{equation}


By using \eqref{mainlemma3: difference between Y and X} and \eqref{mainlemma3: difference between Y and tilde Y}, we obtain
\begin{equation} \label{Y_tilde_approx}
    \sup_{t\in[0,\ep^{-\alpha}]}  |X_1(t) - \tilde Y(t)| \le 3.
\end{equation}

Let us now evaluate the oscillation of the function $\tilde Y$ on short intervals. 
Since there are no short steps, we have
\[
   \tilde\tau_{j+1} - \tilde\tau_{j} >  (\tau_{j+1} -\ep^\beta) - \tau_{j} > \ep^\beta,
\]
on every time interval of length $2\ep^\beta$ the function $\tilde Y$ has at most one jump.
From \eqref{jumpsize2} it follows that
\begin{equation} \label{Y_tilde_oscillation}
    \sup_{0\le s,t\le \ep^{-\alpha}, \, |s-t|\le 2\ep^\beta } |\tilde Y(s)-\tilde Y(t)|\le 2.
\end{equation}

We denote by $\tilde\YY_\ep$ the finite net of step functions constructed in Lemma \ref{sublemma3}.
\medskip

It remains to transfer the obtained results from the time interval $[0,\ep^{-\alpha}]$ 
to the time interval $[0,1]$. We define the required net as
\[
   \YY_\ep := \{ y\in \D, y(\cdot) = \ep \tilde Y( \ep^{-\alpha}\cdot), 
   \tilde Y\in\tilde\YY_\ep \}. 
\]

Let us rewrite \eqref{Y_tilde_approx}, multiplying by $\ep$ and changing time linearly, as 
\begin{equation}\label{Y_tilde_approx2}
    \sup_{t\in[0,1]}| \ep \cdot X_1(t\cdot \ep^{-\alpha}) -
    \ep \cdot \tilde Y(t\cdot\ep^{-\alpha})| 
    \le 3 \, \ep.
\end{equation}
        
Define approximation  $y_\ep\in \YY_\ep$ as
\[
    y_\ep(t):= \ep \cdot  \tilde Y(t\cdot \ep^{-\alpha}),
\]
from \eqref{mainlemma3: two conditions} and \eqref{Y_tilde_approx2} we derive
\begin{equation*}
    \prob \left( \sup_{t\in[0,1]}
      |\ep \cdot X_1(t\cdot \ep^{-\alpha}) - y_\ep(t)|
       > 3 \,\ep
           \right) 
    \le \de,
\end{equation*}
hence by \eqref{mainlemma3: self-similarity},
\begin{equation}\label{mainlemma3: final1}
    \prob \left(\sup_{t\in[0,1]} |X_\ep(t) - y_\ep(t)| > 3\, \ep \right) 
    \le \de.
\end{equation}
Further, the oscillation bound \eqref{Y_tilde_oscillation} yields
\begin{eqnarray*} 
     \omega(y_\ep,2\ep^{B}) &=&  \omega(y_\ep,2\ep^{\beta+\alpha}) 
     =  \sup_{0\le s,t\le 1, \, |s-t|\le 2\ep^{\beta+\alpha}} |y_\ep(s)-y_\ep(t)|
\\    
        &=& \ep \sup_{0\le s,t\le \ep^{-\alpha}, \, |s-t|\le 2\ep^\beta } |\tilde Y(s)-\tilde Y(t)|
        \le 2\ep,
\end{eqnarray*}
thus
\begin{equation}  \label{mainlemma3: final2} 
     \max_{y \in \YY_\ep } \   \omega(y,2\ep^{B}) \le   2\ep.
\end{equation}

       
Lemma \ref{sublemma3} provides the required bound for size of the approximating net $\YY_\ep$.
The probability bound announced in Lemma \ref{mainlemma3} follows from \eqref{mainlemma3: final1},
while the required bound for the oscillation modulus was obtained in \eqref{mainlemma3: final2}.
Therefore, Lemma \ref{mainlemma3} is proved completely.

\section{Further research directions}
\subsection{$\alpha$-stable L\'evy processes with $\alpha\ge 1$}

Our upper bound for $mm$-entropy from \eqref{th1_bounds} actually extends to the case $\alpha\ge 1$. The proof is identical to the provided one up to the standard minor changes in decomposition \eqref{split}. However, we are not at all sure that the order of this bound remains optimal for $\alpha\in [1,2]$.

As for the upper bound in \eqref{th1_bounds}, our proof depends heavily on the assumption $\alpha<1$. It is certainly too crude for possible extension to a larger range of $\alpha$.

As suggests a result  from \cite{VerLif} concerning $mm$-entropy of Wiener process (i.e. $\alpha=2$) in $\C[0,1]$,
the true order of $mm$-entropy is somewhere between $\ep^{-\alpha}$ and 
$\ep^{-\alpha}|\ln\ep|$.

This problem remains open so far.

\subsection{Two-parametric setting}
The results provided in this article concern the case $\ep\to 0$ with $\delta$ fixed.
It is quite  natural to consider the behavior of $H^{mm}(\ep, \de)$ when both $\ep,\delta \to 0$. As the findings of
\cite{VerLif} suggest, there should be several zones defined by relations between $\ep$ and $\delta$
with different asymptotics of $H^{mm}(\ep, \de)$ in each zone. 

Our findings in this direction will be published elsewhere.



\begin{thebibliography}{9}

{\baselineskip=10pt \small


\bibitem{Bil}
P. Billingsley,  Convergence of probability measures. Wiley, New York, 1968.

\bibitem{Rudolf}
R. Clausius,  Ueber verschiedene für die Anwendung bequeme Formen der Hauptgleichungen der mechanischen W\"armetheorie, \emph{Annalen der Physik und Chemie}, 1865, {\bf 125}, 7, 353--400.

\bibitem{Der_dis}
S. Dereich, High resolution coding of stochastic processes and small ball probabilities. Ph.D. dissertation, Technische Univ. Berlin, 2003.

\bibitem{Der03}
S. Dereich,  Small ball probabilities around random centers of Gaussian measures and applications to quantization.  \emph{J. Theoret. Probab.}, 2003, {\bf 16}, 427--449.

\bibitem{DerSch}
S. Dereich, F. Fehringer, A. Matoussi, and M. Scheutsow, On the link between small ball probabilities and the quantization problem for Gaussian measures on Banach spaces. \emph{J. Theoret. Probab.}, 
2003, {\bf 16}, 249--265.

\bibitem{DerLif}
S. Dereich, M. Lifshits, Probabilities of randomly centered small balls and quantization in Banach spaces, 
\emph{Ann. Probab.}, 2005, {\bf 33}, 1397--1421.

\bibitem{Etemadi}
    N. Etemadi, On some classical results in probability theory, \emph{Sankhya Ser. A}, 
    1985, {\bf 47}, 2, 215--221.

\bibitem{Gromov}
M. Gromov, Metric structures for Riemannian and non-Riemannian spaces, Birkh\"user Boston, 1999.

\bibitem{Halmosh}
P.R. Halmos, Measure theory, D. Van Nostrand Company, New York, 1950.

 \bibitem{Kolmogorov}
A.N. Kolmogorov, On certain asymptotic characteristics of completely bounded metric spaces, 
\emph{Dokl. Akad. Nauk}, 1956, {\bf 108}, 385--389 (in Russian).

\bibitem{KolmogorovTihomirov}
A.N. Kolmogorov, V.M. Tikhomirov, $\ep$-entropy and $\ep$-capacity of sets in function spaces, 
\emph{Amer. Math. Soc. Transl. Ser.}, 1961, {\bf 17}, Amer. Math. Soc., Providence, RI, 227--364.



\bibitem{Lifshits1}
M.A. Lifshits, Gaussian random functions, Kluwer, Dordrecht, 1995.

\bibitem{Lifshits2}
M.A. Lifshits, Lectures on Gaussian processes, Springer, Heidelberg, 2012.

\bibitem{LuschPa04}
H. Luschgy, G. Pag\`es,  Sharp asymptotics of the functional quantization problem for Gaussian processes,
\emph{Ann. Probab.}, 2004, {\bf 32}, 1574--1599.

\bibitem{PinSof}
M.S. Pinsker, L.B. Sofman, ($\ep,\delta$)-entropy of completely ergodic stochastic processes, 
\emph{Problems Inform. Transmission}, 1986, {\bf 22}, 251--255.


\bibitem{Shannon}
C.E. Shannon, A mathematical theory of communication, I,II, \emph{Bell Syst. Techn. J.}, 
1948, \textbf{27}, 3, 379--423 ; 4, 623--656.

\bibitem{Skor1}
A.V. Skorokhod. Limit theorems for stochastic processes, \emph{Theor. Probab. Appl.},
1956, {\bf 1}, 261--290.

\bibitem{Skor2}
A.V. Skorokhod. Limit theorems for stochastic processes with independent increments, 
\emph{Theor. Probab. Appl.}, 1957, {\bf 2}, 138--171.

\bibitem{Tih1}
V.M. Tikhomirov, $\ep$-entropy and $\ep$-capacity, in: A.N.Kolmogorov. Information Theory and the Theory of Algorithms, Nauka, Moscow, 1987, 262--269 (in Russian).

\bibitem{Tih2}
V.M. Tikhomirov, Kolmogorov’s work on $\ep$-entropy of functional classes and the superposition of functions, \emph{Russian Math. Surveys}, 1963, {\bf 18}, 51--87.

\bibitem{Tih3}
V.M. Tikhomirov, Widths and entropy, \emph{Russian Math. Surveys}, 1983, {\bf 38},  101--111.

\bibitem{Ver10}
A.M. Vershik, Dynamics of metrics in measure spaces and their asymptotics invariants,
\emph{Markov Process and Related Fields}, 2010, {\bf 16}, 1, 169--185.

\bibitem{VerLif}
A.M. Vershik, M.A.Lifshits. On $mm$-entropy of a Banach space with a Gaussian measure. 
\emph{Theor. Probab. Appl.}, 2023, {\bf 68}, 3, 431--439.

\bibitem{VerVepZat}
A.M. Vershik, G.A. Veprev, and P.B. Zatitsky, Dynamics of metrics in measure spaces and scaling entropy, \emph{Russian Math. Surveys}, 2023, {\bf 78}, 3, 443--499.


\bibitem{Whitt}
W. Whitt. Stochastic-process limits. Springer Series in Operation Research. Springer, New York, 2002.

} 
\end{thebibliography}
\end{document}